\documentclass[a4paper]{article}
\usepackage{amsthm,amssymb,amsmath,enumerate,geometry,graphicx}

\newtheorem{definition}{Definition}
\newtheorem{proposition}[definition]{Proposition}
\newtheorem{theorem}[definition]{Theorem}

\newtheorem{lemma}[definition]{Lemma}
\newtheorem{conjecture}[definition]{Conjecture}

\newtheorem{question}[definition]{Problem}

\newcommand{\trsp}[1]{#1^\intercal} 
\newcommand{\id}{\textrm{Id}} 

\newcommand{\ph}[1]{\mathsf{#1}} 

\newcommand*{\LargerCdot}{\raisebox{-0.25ex}{\scalebox{1.5}{$\cdot$}}}

\title{Maximal determinants of combinatorial matrices}
\author{Henning Bruhn \and Dieter Rautenbach}
\date{}

\begin{document}
\maketitle

\begin{abstract}
We prove that $\det A\leq 6^\frac{n}{6}$ whenever $A\in\{0,1\}^{n\times n}$ contains at most~$2n$ ones. 
We also 
prove an  upper bound on the determinant of matrices with the $k$-consecutive ones property, 
a generalisation of the consecutive ones property, where each row is allowed to have up to $k$ blocks of ones. 
Finally, we prove an upper bound on the determinant of a path-edge incidence matrix in a tree
and use that to bound the leaf rank of a graph in terms of its order. 
\end{abstract}

\section{Introduction}

There is a rich tradition of bounding the determinants of matrices with all entries $0$ or $1$ (or $-1$ and $1$). 
Yet, even very simple questions remain open. For instance: 
\begin{question}\label{mainq}
Given a matrix $A\in\{0,1\}^{n\times n}$
with at most $2n$ ones, how large can its determinant be? 
\end{question}
In fact, Problem \ref{mainq} is open for any linear number $O(n)$ of ones in $A$,
and we consider the case of at most $2n$ as a simple and restricted
representative. Furthermore, this case naturally relates to edge-vertex incidence matrices of graphs,
which we exploit below.

We give an answer to this question
as well as to questions about the determinants of similarly simple matrices.

The most prominent question in this area is probably \emph{Hadamard's maximal determinant problem}:
Given $A\in\{-1,1\}^{n\times n}$, how large can $\det A$ be? 
A partial answer lies in \emph{Hadamard's inequality}~\cite{ha}. 
For this, let us denote the $i$th row of a matrix $A\in\mathbb R^{n\times n}$ by $A_{i,\LargerCdot}$,
and similarly we write $A_{\LargerCdot,j}$ for the $j$th column.
Then
\begin{equation}\label{eh}
|\det(A)|\leq \prod\limits_{i=1}^n ||A_{i,\LargerCdot}||_2.
\end{equation}
If $A$ is a $-1/1$-matrix, then (\ref{eh}) implies 
$|\det(A)|\leq n^{n/2}$,
and $A$ is a {\it Hadamard matrix} of order $n$ if this inequality holds with equality.
It is known that a Hadamard matrix of order $n$ can exist only if $n\in \{ 1,2\}\cup 4\mathbb{N}$.
Paley's conjecture~\cite{pa} states that for all these orders, Hadamard matrices actually do exist.
 
Coming back to Problem~\ref{mainq}, 
with the inequality of the arithmetic and geometric mean, it is  not hard to derive from~\eqref{eh} 
that $\det A\leq 2^\frac{n}{2}$ for any $A\in\{0,1\}^{n \times n}$ with 
at most $2n$ ones. 
Ryser improved this estimate, and at the same time extended it to cover 
more general numbers of $1$-entries in $A$:
\begin{theorem}[Ryser~\cite{ry}]\label{rythm}
Let $A\in\{0,1\}^{n\times n}$ be a matrix containing exactly $kn$ ones.
If $n\geq 2$ and $1\leq k\leq \frac{n+1}{2}$ then, for  $\lambda=k(k-1)/(n-1)$, it follows that 
\[
|\det A|\leq k(k-\lambda)^{\frac{1}{2}(n-1)}.
\] 
\end{theorem}
Moreover, Ryser showed that the bound is tight for many pairs of $k,n$; see next section.
Inspecting Ryser's bound, we see that in the situation of Question~\ref{mainq}, when $k\leq 2$, 
it yields a bound of $\det A\leq 2(2-o(1))^\frac{n-1}{2}$, which is not much better than the 
bound of $2^\frac{n}{2}$ that we get from Hadamard's inequality. 
Given that Ryser's bound is tight for many $k,n$, 
does that mean that we cannot hope for a better bound? No, it turns out. 
Our main result is:

\begin{theorem}\label{2nboundthm}
If $A\in \{ 0,1\}^{n\times n}$ has at most $2n$ non-zero entries, then 
$|\det(A)|\leq 2^{n/6}\cdot 3^{n/6}\approx 1.348^n.$
\end{theorem}

While a substantial improvement on the bound of $2^\frac{n}{2}\approx 1.414^n$,
the bound in Theorem~\ref{2nboundthm} 
is almost certainly not best possible. Indeed, the best lower bound we have found 
comes from matrices of the kind
\[
A=\text{diag}(C,\ldots,C),\text{ where }C=\begin{pmatrix}1 & 1 & 0\\0& 1 & 1\\ 1& 0 & 1\end{pmatrix},
\]
that have determinant $2^\frac{n}{3}\approx 1.260^n$. We believe 
that larger determinants are impossible:
\begin{conjecture}\label{conj2nentries}
If $A\in \{ 0,1\}^{n\times n}$ has at most $2n$ non-zero entries, then $|\det(A)|\leq 2^{n/3}$.
\end{conjecture}

We discuss the conjecture a bit more in the next section, and 
we prove Theorem~\ref{2nboundthm} in Sections~\ref{gramsec} and~\ref{sparsesec}. 
As a tool in the proof we consider incidence matrices $I$ 
of graphs. Somewhat strikingly, we find that  the determinant $\det(I\trsp I)$
of the Gramian of the incidence
matrix of a tree does not depend on the structure of the tree (Theorem~\ref{theoremgramtree}).

\medskip
We also consider two other types of combinatorial matrices. Many $0/1$-matrices arise
in combinatorial optimisation as constraints matrices of (mixed) integer programs. 
A particularly nice type of such a matrix are those with the \emph{consecutive ones property}.
These are $0/1$-matrices, in which in each row all the one entries are consecutive. 
It is a standard exercise to show that matrices with the consecutive ones property are 
totally unimodular, which means that  solving the linear relaxation also solves the 
corresponding integer program. 

We consider two types of generalisations of the consecutive ones property. In Section~\ref{consec},
we treat matrices with the \emph{$k$-consecutive ones property}: matrices in whose rows the 
$1$-entries appear in at most $k$ blocks. We will prove an upper bound on the determinant of such 
matrices (Theorem~\ref{theoremkcop}).
Perhaps not immediately recognisable as a generalisation 
are  path-edge incidence matrices 
of a tree. We treat these in Section~\ref{pathsec} and similarly prove an upper bound on the determinant (Theorem~\ref{twothm}). 
Finally, in Section~\ref{sec_lr}, we use the bound  
of Theorem~\ref{twothm}
in order to upper-bound a somewhat complicated graph parameter, 
the \emph{leaf rank} of a graph. Our result is the first that bounds the 
leaf rank in terms of the order of the graph.


\section{Discussion}

Ryser's inequality in Theorem~\ref{rythm} is tight in many circumstances. 
Ryser proved that equality holds if and only if  $A$ is the incidence matrix of an {\it $(n,k,\lambda)$-configuration},
that is, when $A$ is the incidence matrix of  a collection of $n$ subsets $S_1,\ldots,S_n$ of $[n]$ such that 
$|S_i|=k$ and $|S_i\cap S_j|= \lambda$ for every $i\in [n]$ and every $j\in [n]\setminus \{ i\}$.
(We write $[n]$ as a shorthand for $\{1,\ldots, n\}$.)

As a consequence, equality in Theorem~\ref{rythm} may only hold if $\lambda$ is an integer, which implies $\lambda\geq 1$.  
That, in turn, entails $k\geq \sqrt n$. To sum up, Ryser's theorem can only be tight if the matrix $A$
contains at least $n\sqrt n$ many ones. This explains why for only $2n$ ones, or indeed any linear number of ones, 
improvements may be possible.

There is another well-known bound on the determinant of a $0/1$-matrix that comes from a
very elegant transformation that turns any matrix $A\in\{0,1\}^{n \times n}$ into 
a matrix $B\in\{-1,1\}^{(n+1)\times (n+1)}$ such that $2^n\det A=\det B$;   see~\cite{bc}.
In particular, if $A$ is any $0/1$-matrix (with possibly more than $2n$ ones), then 
direct application of Hadamard's inequality~\eqref{eh} to $A$
yields $|\det A|\leq n^\frac{n}{2}$, while an application of the inequality to $B$ yields
the better bound of $|\det(A)|\leq 2^{-n}(n+1)^{(n+1)/2}$. 
In our setting, however, where $A$ contains at most $2n$ ones, the latter bound is actually worse
than what we can get from~\eqref{eh} by exploiting that, on average, every row of $A$ has Euclidean norm at most~$\sqrt 2$.

\medskip
Conjectures need evidence. As a first piece of evidence for Conjecture~\ref{conj2nentries},
we quickly prove a special case, where, more strongly, we require that every row 
of the matrix contains at most two ones:

\begin{proposition}\label{proposition2perrow}
If $A\in \{ 0,1\}^{n\times n}$ and if every row of $A$ contains at most two ones, then $|\det(A)|\leq 2^{n/3}.$
\end{proposition}
\begin{proof}
Clearly, we may assume that $A$ is non-singular.
The proof is by induction on~$n$. 
For $n\leq 3$, the bound is easily verified.
Now, let $n>4$.
If some row or column of $A$ contains at most one non-zero entry, 
then expanding the determinant of $A$ along this row or column,
we obtain, by induction, $|\det(A)|\leq 2^{(n-1)/3}$.
Hence, every row and column of $A$ contains exactly two non-zero entries.
Since $A$ is non-singular, this implies that, suitably permuting the rows and columns of $A$, we may assume that 
$$A=
\begin{pmatrix}
A_1 & 0\\
0 & A_2
\end{pmatrix},
$$
where $A_1$ is the vertex-edge incidence matrix of a cycle $C_\ell$.
If $\ell$ is even, then $\det(A_1)=0$, and, if $\ell$ is odd, then $|\det(A_1)|=2$.
Since $A$ is non-singular, we obtain that $\ell$ is an odd integer at least $3$.
By induction, we obtain
$$|\det(A)|=|\det(A_1)|\cdot |\det(A_2)|\leq 2\cdot 2^{(n-\ell)/3}\leq 2^{n/3},$$
which completes the proof. 
\end{proof}

Our main result is to improve on Ryser's theorem for matrices that contain up to~$2n$ ones.
But what about matrices with more ones? What about matrices with up to $3n$
ones? While we do not know how to prove a small bound on the determinant, we at least can offer
a lower bound that arises in a similar way as for $2n$ ones. For this, denote by $C\in\{0,1\}^{7\times 7}$ 
the incidence matrix of the Fano plane, or a $(7,3,1)$-configuration in Ryser's terminology. 
Then $\det C=3\cdot 2^\frac{7-1}{2}=24$, by Ryser's theorem (or elementary calculation). Thus 
the $n\times n$-matrix 
$A=\text{diag}(C,\ldots,C)$ has determinant $24^\frac{n}{7}\approx 1.57^n$. We suspect that 
this is the maximal determinant for $3n$ ones.

\section{Gramians of graphs}\label{gramsec}

In this section we prove the main lemma for the proof of Theorem~\ref{2nboundthm}:

\begin{lemma}\label{2nlem}
Let $I\in\{0,1\}^{m\times n}$ be a matrix that contains exactly two ones in each row. 
Then 
$$
\det (I\trsp I)\leq
\begin{cases}
2^m &\mbox{if $m\leq \frac{n}{2}$ and}\\
2^{\frac{1}{3}(n+m)} &\mbox{if $m\geq \frac{n}{2}$.}
\end{cases}
$$
\end{lemma}

We recall  some basic properties about the \emph{Gramian} $I\trsp I$
that hold for any   matrix $I\in\mathbb{R}^{m\times n}$.
These properties follow quite easily from properties of determinants:
If $J$ arises from $I$ by
\begin{itemize}
\item[(G$_{\rm 1}$)] either multiplying some row with $-1$, 
\item[(G$_{\rm 2}$)] or adding a multiple of some row to another row,
\item[(G$_{\rm 3}$)] or permuting rows and columns,
\end{itemize}
then $\det (J\trsp J)=\det (I\trsp I)$.
Consider, for instance, the situation that $J$ arises from $I$ 
by adding $\alpha I_{j,\LargerCdot}$ to $I_{i,\LargerCdot}$ for distinct indices $i$ and $j$, and some real $\alpha$. If $A=I\trsp I$, and $B$ arises from $A$ 
by adding $\alpha A_{j,\LargerCdot}$ to $A_{i,\LargerCdot}$,
then $J\trsp J$ arises from $B$ 
by adding $\alpha B_{\LargerCdot,j}$ to $B_{\LargerCdot,i}$.
Since both operations do not change the determinant, (G$_{\rm 2}$) follows.

\medskip
If $I\in\{0,1\}^{m\times n}$ is a $0/1$-matrix with exactly two ones in each row, so if $I$ is as in Lemma~\ref{2nlem},
then in particular $I$ may be seen as the \emph{edge-vertex incidence} matrix of a graph $G$
with vertex set $\{v_1,\ldots,v_n\}$ and edge set $\{e_1,\ldots,e_m\}$ such that $v_i$ is incident with $e_j$
precisely when $I_{j,i}=1$. 
We find this view quite convenient as many of the arguments and statements 
we will need are easier to phrase in the language of graphs. 
For all basic graph-theoretical definitions and concepts we refer to the textbook of Diestel~\cite{Rhd}. 
We note that all our graphs will be finite, simple and undirected. 

There is a long interest in the algebraic properties of various matrices associated with graphs. 
For instance, 
Mowshowitz~\cite{Mow} showed that the characteristic polynomial of the adjacency matrix of a tree
can be expressed in terms of matchings. The matrix $I\trsp I$, with $I$ being the incidence matrix of the graph $G$,
is closely related to the \emph{Laplacian} of a graph. Indeed, if $Q$ is the incidence matrix of an \emph{orientation} of $G$,
then the Laplacian of $G$ is defined as $L=\trsp QQ$. It is not hard to verify that $L$ does not depend on 
the actual orientation. We refer to the book of Chung~\cite{chung} for more on the Laplacian
and other topics in spectral graph theory. 
While the Laplacian has rows and columns that correspond to the vertices of $G$, the 
rows and columns of the matrix $I\trsp I$ correspond to the edges of $G$. Thus, the matrix $I\trsp I$
that we are interested in, is even more closely related to the \emph{edge-Laplacian}, defined as $Q\trsp Q$,
of the graph. For some results about the edge-Laplacian see Bapat~\cite{Bap}.

Often  the determinant of a matrix associated with a graph $G$
 encodes properties of $G$. In contrast, 
we find that, for a tree, the determinant $\det (I\trsp I)$ does not depend on the structure of $T$ at all. 

\begin{theorem}\label{theoremgramtree}
If $T$ is a tree of order $n\geq 2$, and $I\in \{ 0,1\}^{E(T)\times V(T)}$ 
is the edge-vertex incidence matrix of $T$, then $\det (I\trsp I)=n$. 
\end{theorem}
\begin{proof} 
For a function $s:V(T)\to \{ -1,1\}$,
let $I_s$ arise from $I$ by multiplying the column with index $u$ by $s(u)$ for every vertex $u$ of $T$.
We say that $s$ is {\it special}, if $s(u)=1$ for every non-leaf $u$ of $T$.
Note that $I_s\trsp I_s=I\trsp I$ if $s$ is special.
We prove the following slightly more general statement:
{\it If $T$ is a tree of order $n$, $s:V(T)\to \{ -1,1\}$ is special, and $I$ is the edge-vertex incidence matrix of $T$,
then $\det (I_s\trsp I_s)=n$.}

First, if $T$ is a star, then 
$$\det (I_s\trsp I_s)
=
\det\begin{pmatrix}
2 & 1 & 1& \ldots &1 & 1\\
1 & 2 & 1 & \ldots & 1 & 1\\
1 & 1 & 2 & \ldots & 1& 1\\
\vdots & \vdots  & \vdots  & \ddots & \vdots  & \vdots \\
1& 1& 1& \ddots &2 & 1\\
1 & 1 & 1& \ldots &1 & 2
\end{pmatrix}
=
\det\begin{pmatrix}
2 & 1 & 1& \ldots &1 & 1\\
-1 & 1 & 0 & \ldots & 0 & 0\\
0 & -1 & 1 & \ldots & 0& 0\\
\vdots & \vdots  & \vdots  & \ddots & \vdots  & \vdots \\
0& 0& 0& \ddots &1 & 0\\
0 & 0 & 0& \ldots &-1 & 1
\end{pmatrix}
=n,$$
where the last equality follows easily by induction 
expanding the determinant along the final column.
(Alternatively, the result also follows very quickly with the help of Sylvester's determinant
identity.)

Next, we consider a simple operation moving a leaf within a tree.
More precisely,
let $T$ be such that $uvw$ is a path in $T$, where $u$ is a leaf, and $w$ is not a leaf.
Let $T'=T-uv+uw$, and let 
$$s':V(T')\to \{ -1,1\}:x\mapsto 
\begin{cases}
-s(x)&\mbox{, if $x=u$, and}\\
s(x)&\mbox{, if $x\not=u$.} 
\end{cases}$$
Clearly, since $s$ is special, the function $s'$ is special.
Furthermore, if $I'$ is the edge-vertex incidence matrix of $T'$, 
then $I'_{s'}$ arises from $I_s$ 
by subtracting the row with index $vw$ from the row with index $uv$,
and multiplying the new row with index $uv$ with $-1$,
which, by (G$_{\rm 1}$) and (G$_{\rm 2}$), implies 
$\det ({I'}_{s'}\trsp {I'}_{s'})=\det (I_s\trsp I_s)$.

Since every tree can be transformed to a star by a sequence of such simple operations,
the desired statement follows.
\end{proof}
\noindent Let $G$ be a connected graph of order $n$ and size $m$,
and let $I$ be the edge-vertex incidence matrix of $G$.
If $m>n$, then the rows of $I$ are linearly dependent, which, by (G$_{\rm 2}$),
easily implies $\det (I\trsp I)=0$.
If $m=n$, then $G$ arises by attaching (possibly trivial) trees to the vertices of some cycle $C$. 
Iteratively expanding the determinant of the square matrix $I$ 
along columns corresponding to vertices of degree $1$,
we obtain that $\det (I)=\det (J)$,
where $J$ is the edge-vertex incidence matrix of $C$.
It is a simple known fact that, 
if $C$ has even order, then $\det (J)=0$, and, 
if $C$ has odd order, then $|\det (J)|=2$.

Altogether, we obtain
\begin{eqnarray}\label{egram}
\det (I\trsp I)=
\begin{cases}
n&\mbox{if $G$ is a tree,}\\
0&\mbox{if $m>n$ or $G$ contains an even cycle, and}\\
4&\mbox{if $m=n$ and $G$ contains an odd cycle.}
\end{cases}
\end{eqnarray}

Our next result is Lemma~\ref{2nlem} phrased in terms of graphs.

\medskip

\noindent {\bf Lemma~\ref{2nlem}b}
{\it If $G$ is a graph of order $n$ that has at least one and at most $m$ edges, 
and $I$ is the edge-vertex incidence matrix of $G$, then
$$
\det (I\trsp I)\leq
\begin{cases}
2^m &\mbox{if $m\leq \frac{n}{2}$ and}\\
2^{\frac{1}{3}(n+m)} &\mbox{if $m\geq \frac{n}{2}$.}
\end{cases}
$$}
\begin{proof}
Among all graphs with $n$ vertices and  at most $m$ edges (but at least one),
let $G$ be chosen in such a way that $\det (I\trsp I)$ is largest possible,
and, subject to this condition, the number of edges of $G$ is smallest possible.
Clearly $\det (I\trsp I)>0$.
Let $G$ have $k$ components $G_1,G_2,\ldots,G_k$
with edge-vertex incidence matrices $I_1,I_2,\ldots,I_k$.
Since, by (G$_{\rm 3}$), 
 permuting rows and columns does not affect $\det (I\trsp I)$,
we may assume that $I\trsp I$ is a block diagonal matrix
with $k$ blocks each corresponding to one component. Then
$$\det (I\trsp I)=\det (I_1\trsp I_1)\cdot \det (I_2\trsp I_2)\cdots\det (I_k\trsp I_k).$$

We claim that
\begin{enumerate}[\rm (i)]
\item every non-tree component of $G$ is a triangle, and
\item every tree component of $G$ is either an isolated vertex, or an edge, or a path
on three vertices. 
\end{enumerate}
Indeed, let $J$ be the edge-vertex incidence matrix of a triangle.
Since, again by~\eqref{egram}, 
$\det (I_i\trsp I_i)$ equals $\det (J\trsp J)$ 
if $G_i$ is a non-tree component, 
and since $G$ has as few edges as possible,
we obtain~(i). Suppose that $G$ has a tree component of at least $r\geq 4$ vertices. Then 
replacing the component by two components, one  a tree on $r-2$ vertices and 
the other an edge, results, by~\eqref{egram}, in an overall larger determinant. This proves~(ii).

For $i\in [3]$, let $k_i$ be the number of components of $G$ with exactly $i$ edges.
It follows that 
\begin{eqnarray}\label{emstar}
\det (I\trsp I) & =& 2^{k_1}\cdot 3^{k_2}\cdot 4^{k_3}
= 2^{k_1+\log_2(3)k_2+2k_3},
\end{eqnarray}
and that the non-negative integers $k_1$, $k_2$, and $k_3$ satisfy 
\begin{eqnarray}
k_1+2k_2+3k_3 & \leq & m\label{elin1}\mbox{ and}\\
2k_1+3k_2+3k_3 & \leq & n\label{elin2}.
\end{eqnarray}
Since $\log_2(3)\leq \frac{5}{3}\leq 2$, we obtain
\begin{align*}
\log_2\left(\det (I\trsp I)\right)&\stackrel{(\ref{emstar})}{\leq} k_1+2k_2+3k_3 \stackrel{(\ref{elin1})}{\leq} m,\mbox{ and }\\[3mm]
\log_2\left(\det (I\trsp I)\right)&\stackrel{(\ref{emstar})}{\leq}
k_1+\frac{5}{3}k_2+2k_3\\
&= \frac{1}{3}\left(k_1+2k_2+3k_3\right)+\frac{1}{3}\left(2k_1+3k_2+3k_3\right)
\stackrel{(\ref{elin1}),(\ref{elin2})}{\leq} \frac{m+n}{3},
\end{align*}
which completes the proof.
\end{proof}

Lemma~\ref{2nlem}b is tight.
In fact, 
if $m\leq \frac{n}{2}$, then $G=mK_2\cup (n-2m)K_1$ yields equality, and, 
if $\frac{n}{2}\leq m\leq n$ and $2m-n$ is divisible by $3$, 
then $G=(n-m)K_2\cup \frac{1}{3}(2m-n)K_3$ yields equality.

\section{Proof of Theorem~\ref{2nboundthm}}\label{sparsesec}

For the  proof of Theorem~\ref{2nboundthm}, we need another property of the Gramian.
Let $B={A\choose a}$ be a real matrix with last row $a$. 
If $a'$ is the orthogonal projection of $a$ onto the row space of $A$, 
then, by (G$_{\rm 2}$), 
\begin{eqnarray}
\det (B\trsp B)&=&\det \left({{A\choose a-a'}}\trsp {{A\choose a-a'}}\right)
=\det 
\begin{pmatrix}
A\trsp A & 0 \nonumber\\
0 & ||a-a'||_2^2
\end{pmatrix}\\
&=&\det (A\trsp A)\cdot ||a-a'||_2^2\nonumber \\
&\leq &\det (A\trsp A)\cdot ||a||_2^2.\label{eg4}
\end{eqnarray}

\begin{proof}[Proof of Theorem~\ref{2nboundthm}]
Clearly, we may assume that $A$ is non-singular; 
in particular, every row of $A$ contains at least one $1$-entry.

Iteratively, expand the determinant $\det A$ over rows with only one $1$-entry, and 
let $n_1$ be the number of rows eliminated in this way, and let $A'$ be the 
resulting $(n-n_1)\times (n-n_1)$-matrix with $|\det(A)|=|\det(A')|$. 
In particular, $A'$ contains at most $2n-n_1$ many ones, and every row of $A'$ contains at least two ones. 
Denote by $A_2$ the matrix consisting of the rows of $A'$ with exactly two ones, 
and let $A_3$ be the matrix consisting of the other rows (each with at least three ones).
Let $A_i$ have $n_i$ rows for $i=2,3$.
Since $n_1+2n_2+3n_3\leq 2n=2n_1+2n_2+2n_3$, we obtain $n_3\geq n_1$. 

Iteratively applying (\ref{eg4}), we obtain 
\begin{eqnarray*}
\det (A)^2&\leq& 
\det \left(A_2\trsp A_2\right)\cdot \prod_{i=1}^{n_3}||A_{3,\LargerCdot}||^2_2.
\end{eqnarray*}

Noting that $A_2$ has $n-n_1$ many columns and $n_2$ many rows, 
we apply Lemma~\ref{2nlem} to get 
$$\det (A_2\trsp A_2)\leq 
\begin{cases}
2^{n_2} &\mbox{if $n_2\leq \frac{n-n_1}{2}$ and}\\
2^{\frac{1}{3}((n-n_1)+n_2)} &\mbox{if $n_2\geq \frac{n-n_1}{2}$.}
\end{cases}
$$
Since $A_3$ is a $0/1$-matrix with at most $2n-n_1-2n_2=n_1+2n_3$ non-zero entries, 
the inequality between the geometric and the arithmetic mean implies
\begin{align}\label{A3est}
\prod_{i=1}^{n_3}||A_{3,\LargerCdot}||^2_2 & \leq  \left(\frac{1}{n_3}\sum_{i=1}^{n_3}||A_{3,\LargerCdot}||^2_2\right)^{n_3}
\leq \left(\frac{2n-n_1-2n_2}{n_3}\right)^{n_3}
=\left(2+\frac{n_1}{n_3}\right)^{n_3}.
\end{align}
Using $n=n_1+n_2+n_3$,  
we obtain
\begin{eqnarray*}
\det (A)^2& \leq & 
\begin{cases}
2^{n-n_1-n_3}\cdot \left(2+\frac{n_1}{n_3}\right)^{n_3} &\mbox{if $n_1+2n_3\geq n$, and}\\[3mm]
2^{\frac{1}{3}(2n-2n_1-n_3)}\cdot \left(2+\frac{n_1}{n_3}\right)^{n_3} &\mbox{if $n_1+2n_3\leq n$.}
\end{cases}
\end{eqnarray*}
Setting $x=\frac{n_1}{n}$ and $y=\frac{n_3}{n}$,
we rewrite the right-hand side as $2^{f(x,y)n}$, where
\begin{eqnarray*}
f(x,y) &=&
\begin{cases}
1-x-y+\frac{1}{\ln(2)}\cdot y\cdot \ln\left(2+\frac{x}{y}\right)
&\mbox{if $x+2y\geq 1$, and}\\[3mm]
\frac{1}{3}(2-2x-y)+\frac{1}{\ln(2)}\cdot y\cdot \ln\left(2+\frac{x}{y}\right)
&\mbox{if $x+2y\leq 1$.}
\end{cases}
\end{eqnarray*}
Since $n_3\geq n_1$, we have $0\leq y\leq x\leq 1$.

Let 
\begin{align*}
M_1 & =  \max\left\{ f(x,y):0\leq y\leq x\leq 1\mbox{ and }x+2y\geq 1\right\}\mbox{ and}\\
M_2 & =  \max\left\{ f(x,y):0\leq y\leq x\leq 1\mbox{ and }x+2y\leq 1\right\}.
\end{align*}
Taylor expansion around $z=1$ implies $\ln(2+z)\leq \ln(3)+\frac{z-1}{3}$ for $z\geq 0$, and, hence,
\begin{align*}
M_1 &\leq  
\max\Big\{
\underbrace{1-x-y+\frac{y}{\ln(2)}\left(\ln(3)+\frac{\tfrac{x}{y}-1}{3}\right)}_{=:g_1(x,y)}
:0\leq y\leq x\leq 1,\,x+2y\geq 1\Big\}\\
M_2 &\leq  
\max\Big\{
\underbrace{\tfrac{1}{3}(2-2x-y)+\frac{y}{\ln(2)}\left(\ln(3)+\frac{\tfrac{x}{y}-1}{3}\right)}_{=:g_2(x,y)}
:0\leq y\leq x\leq 1,\,x+2y\leq 1\Big\}.
\end{align*}
Since $\frac{\partial}{\partial x}g_1(x,y)$ is a negative constant 
and $\frac{\partial}{\partial y}g_2(x,y)$ is a positive constant, we obtain
\begin{align*}
M_1 &\leq  \max\left\{g_1\Big(\max\Big\{ y,1-2y\Big\},y\Big): 0\leq y\leq 1\right\}\mbox{, and}\\
M_2 &\leq  
\max\left\{g_2\left(x,\min\left\{ x,\frac{1-x}{2}\right\}\right):0\leq x\leq 1\right\}.
\end{align*}
Since 
$\frac{\partial}{\partial y}g_1(y,y)$ is a negative constant and
$\frac{\partial}{\partial y}g_1(1-2y,y)$ is a positive constant,
we obtain $M_1\leq g_1(1/3,1/3)=\frac{1}{3}+\frac{\ln(3)}{3\ln(2)}$.
Since $\frac{\partial}{\partial x}g_2(x,x)$ is a positive constant and
$\frac{\partial}{\partial x}g_2\left(x,\frac{1-x}{2}\right)$ is a negative constant,
we obtain $M_2\leq g_2(1/3,1/3)=\frac{1}{3}+\frac{\ln(3)}{3\ln(2)}$.
Altogether, we obtain
$|\det (A)|\leq 2^{ \left(\frac{1}{3}+\frac{\ln(3)}{3\ln(2)}\right)\frac{n}{2}}=2^{n/6}\cdot 3^{n/6}\approx 1.348^n$.
\end{proof}

Obviously, we do not believe the bound in the theorem to be best possible, and indeed
there are some ways to improve the bound further. Since we can assume that every column
with a~$1$ in some row of~$A_2$ or of $A_3$ contains at least two ones (otherwise we may expand over the column),
we see that for every row $a$ in $A_2$ or in $A_3$ there is some other row $b$ with scalar product $a\trsp b\geq 1$.
Then, however, one of the two rows $a,b$ can be replaced by a shorter row. In this way, 
it is possible to shorten at least half of the rows in $A_3$ somewhat, which will result in 
a tighter estimation in~\eqref{A3est}. 
A similar approach is used in the proof of Theorem \ref{theorem2cop} below.
We have not worked out the details here, 
because it is clear that this effect will not be enough 
to achieve the bound in Conjecture~\ref{conj2nentries}.

\section{The $k$-consecutive ones property}\label{consec}

A $0/1$-matrix $A$ has the \emph{consecutive ones property} if, in each row of $A$, the 
$1$-entries are consecutive. That is, in each row the ones form a block. 
We generalise this property by allowing more blocks of ones. 
For a positive integer $k$,
let us say 
 that $A$ has the \emph{$k$-consecutive ones property}
if in every row of $A$ the $1$-entries form at most $k$ blocks. 
That means every row has the form
$$0\ldots 0\underbrace{1\ldots 1}_{\mbox{1st block}}
0\ldots 0\underbrace{1\ldots 1}_{\mbox{2nd block}}
0\ldots 0 \ldots 0\ldots 0 
\underbrace{1\ldots 1}_{\mbox{$\ell$th block}}
0\ldots 0$$
for some $\ell\leq k$.
It is well known that matrices with the {consecutive ones property}  are totally unimodular:
 every square submatrix has determinant $-1$, $0$, or $1$.
While the determinant of a matrix with the
$k$-consecutive ones property is no longer bounded by any constant, 
we will still find a bound that is substantially better than Hadamard's bound~\eqref{eh}.

We begin with a simple consequence of Hadamard's inequality:

\begin{proposition}\label{proposition01-1}
If $A\in \{ -1,0,1\}^{n\times n}$ has $t$ non-zero entries, then 
\begin{eqnarray}\label{eprop}
|\det(A)|\leq \left(\frac{t}{n}\right)^{n/2}
\end{eqnarray}
with equality if 
$\frac{t}{n}$ and $\frac{n^2}{t}$ are integers, 
and there is a Hadamard matrix of order $\frac{t}{n}$.
\end{proposition}
\begin{proof}
For $i\in [n]$, let $s_i$ be the sum of the squares of the entries in column $i$ of $A$.
Note that $\sum\limits_{i=1}^n s_i=t$.
By Hadamard's inequality (\ref{eh}), and, since the geometric mean is at most the arithmetic mean, we obtain
\begin{eqnarray*}
|\det(A)| & \leq & \prod\limits_{i=1}^n\sqrt{s_i}
\leq \left(\frac{1}{n}\sum\limits_{i=1}^n s_i\right)^{n/2}
= \left(\frac{t}{n}\right)^{n/2}.
\end{eqnarray*}
If $k=\frac{t}{n}$ and $\frac{n^2}{t}$ are integers, and $H$ is a Hadamard matrix of order $k$,
then the block diagonal matrix $A$ with $\frac{n^2}{t}$ copies of $H$ along the diagonal
satisfies
$$|\det(A)|=|\det(H)|^{n^2/t}
=\left(\left(\frac{t}{n}\right)^{t/2n}\right)^{n^2/t}
=\left(\frac{t}{n}\right)^{n/2},$$
which completes the proof.
\end{proof}

It is conceivable that equality holds in (\ref{eprop}) if and only if 
$\frac{t}{n}$ and $\frac{n^2}{t}$ are integers, and there is a Hadamard matrix of order $\frac{t}{n}$.
Note that, unlike  Ryser's theorem, which has a similar setting,  Proposition~\ref{proposition01-1} is tight for some $n\times n$ matrices with only $O(n)$ non-zero entries.

\begin{theorem}\label{theoremkcop}
If $A\in \{ 0,1\}^{n\times n}$ has the $k$-consecutive ones property, then 
$|\det(A)|\leq (2k)^{n/2}.$
Furthermore, for every $k$ such that a Hadamard matrix of order $k$ exists,
and every $n$ for which $n/k$ is an integer,
there is a matrix $A$ with $|\det(A)|=k^{(n-k)/2}$ that has the $k$-consecutive ones property.
\end{theorem}
\begin{proof}
Let $A\in \{ 0,1\}^{n\times n}$ have the $k$-consecutive ones property.
Let the matrix $B\in \{ -1,0,1\}^{n\times n}$ arise from $A$ by subtracting, for $i$ from $n-1$ down to $1$, column $i$ from column $i+1$.
Since $A$ has the $k$-consecutive ones property, every row of $B$ contains at most $2k$ non-zero entries
that alternate betwee $1$ and $-1$, 
and Proposition \ref{proposition01-1} implies $|\det(A)|=|\det(B)|\leq (2k)^{n/2}$.

Now, let $k$ and $n$ be such that a Hadamard matrix $H$ of order $k$ exists,
and $p=n/k$ is an integer.
Clearly, we may assume that $H_{\LargerCdot,1}$ is the all-$1$-vector.

There is a matrix $R\in \{ -1,0,1\}^{k\times k}$ such that $R_{\LargerCdot,k}$ is the all-$0$-vector, 
and in every row of the matrix 
$(H_{\LargerCdot,1}R_{\LargerCdot,1}H_{\LargerCdot,2}R_{\LargerCdot,2}\ldots H_{\LargerCdot,k}R_{\LargerCdot,k})$
\begin{enumerate}[\rm (i)]
\item there are at most $2k$ non-zero entries,
\item the first non-zero entry is $1$, and 
\item the $1$-entries and $-1$-entries alternate.
\end{enumerate}
In fact, each column $R_{\LargerCdot,i}$ of $R$ with $i\leq k-1$ is uniquely determined by the columns $H_{\LargerCdot,i}$ and $H_{\LargerCdot,i+1}$ of $H$.

Let $\id_k$ be the identity matrix of order $k$, and let the matrix $A$ of order $n=pk$ be as follows.
$$A=
\begin{pmatrix}
H & R & 0 & \cdots & 0& 0\\
0 & H & R & \cdots & 0& 0\\
\vdots & \vdots & \ddots & \ddots & \vdots & \vdots\\
0 & 0 & 0 & \ddots & R & 0\\
0 & 0 & 0 & \cdots & H & R\\
0 & 0 & 0 & \cdots & 0 & \id_k
\end{pmatrix}
$$
Clearly, 
$$|\det(A)|=|\det(H)|^{p-1}=k^{k(p-1)/2}=k^{(n-k)/2}.$$ 
Let $v_i=A_{\LargerCdot,i}$ for $i\in [n]$, and let $B$ be the matrix equal to 
$$
\Big(
v_1v_{k+1}\ldots v_{(p-1)k+1}
v_2v_{k+2}\ldots v_{(p-1)k+2}
\ldots
v_{k-1}v_{k+(k-1)}\ldots v_{(p-1)k+(k-1)}
v_kv_{2k}\ldots v_{pk}\Big).$$
Since $B$ arises by suitably permuting the columns of $A$,
we obtain $|\det(B)|=|\det(A)|$, 
and the properties of $R$ imply that every row of $B$ satisfies~(i)--(iii).
Let $C$ arise from $B$ by adding, for $i$ from $1$ up to $n-1$, column $i$ to column $i+1$.
 
The matrix $C$ has the $k$-consecutive ones property, and satisfies
$|\det(C)|=k^{(n-k)/2}$.
\end{proof}

Unfortunately, there is a notable gap between the upper bound and the determinant 
of the matrix we construct. It is quite likely that that the upper bound can be improved
substantially. 
The following result concerning the $2$-consecutive ones property
illustrates possible ways of improving Theorem \ref{theoremkcop}.

\begin{theorem}\label{theorem2cop}
If $A\in \{ 0,1\}^{n\times n}$ has the $2$-consecutive ones property, then 
$|\det(A)|\leq 3.936^{n/2}.$
\end{theorem}
\begin{proof}
Let $A\in \{ 0,1\}^{n\times n}$ have the $2$-consecutive ones property,
and let $B\in \{ -1,0,1\}^{n\times n}$ arise from $A$ exactly as in the beginning of the proof of Theorem \ref{theoremkcop}.
Note that every row of $B$ contains at most four non-zero entries,
and that, within every row that contains four non-zero entries, the $-1$-entries and the $1$-entries alternate.
Clearly, we may assume that the rows of $B$ are linearly independent, since otherwise $|\det(A)|=|\det(B)|=0$. 
Let $R$ be the set of rows of $B$ that contain four non-zero entries, let $S$ be the set of the remaining rows of $B$,
and let $S'=\emptyset$. \sloppy

Our strategy to improve the bound from Theorem \ref{theoremkcop}
exploits the fact that the rows in $R$ cannot all be pairwise orthogonal if $k=|R|$ is too big. More precisely, we will show that as long as $4k>3n$, there are two rows $a$ and $b$ in $R$ with $a \trsp b\not=0$. We remove $a$ from $R$, add the shorter row $a'=a-\frac{a \trsp b}{b \trsp b}b$ to $S'$, and iterate this process.
Once $4k\leq 3n$, at least $n/4$ rows are in $S\cup S'$.
Since a matrix whose set of rows is $R\cup S\cup S'$
has the same determinant as $B$, 
an application of Hadamard's inequality (\ref{eh}) 
will yield the desired bound.

Therefore, let $4k>3n$.
By the pigeon-hole principle, 
there are four rows $r_1$, $r_2$, $r_3$, and $r_4$ in $R$ that have a non-zero entry in the same column, 
say the column with index $j$.
Suppose, for a contradiction, that $r_i\trsp r_{i'}=0$ for every two distinct indices $i$ and $i'$ in $[4]$.
Since the $-1$-entries and the $1$-entries alternate in each row in $R$,
this implies that,
for every two distinct indices $i$ and $i'$ in $[4]$,
there are exactly two columns in which $r_i$ and $r_{i'}$ both have non-zero entries.
Furthermore, this latter property implies that there are no three rows among $r_1$, $r_2$, $r_3$, and $r_4$ 
that all have non-zero entries in the same two columns.
Since multiplying rows by $-1$ affects 
neither the absolute value of the determinant
nor the pairs of orthogonal rows,
we may assume that the rows $r_1$ to $r_4$ have entry $1$ in column $j$.
Let $j_2$, $j_3$, and $j_4$ be the column indices distinct from $j$ in which row $r_1$ has its non-zero entries
such that the entry in column $j_3$ is $1$.
Since no three rows among $r_1$, $r_2$, $r_3$, and $r_4$ have non-zero entries in the same two columns,
we may assume, by symmetry, that 
row $r_2$ has entry $1$ in column $j_2$,
row $r_3$ has entry $-1$ in column $j_3$, and
row $r_4$ has entry $1$ in column $j_4$.
It follows that in the unique column distinct from columns $j$, $j_2$, and $j_4$, 
in which $r_2$ and $r_4$ both have a non-zero entry, 
both these rows have a $-1$-entry, 
which implies the contradiction that they cannot be orthogonal.

Altogether, it follows that for $4k>3n$,
there are two columns $a$ and $b$ in $R$ that are not orthogonal.
By symmetry, we may assume that $a \trsp b>0$.
Since the rows of $B$ are linearly independent,
and since $a,b$ are integral, we obtain $a \trsp b\geq 1$.
Let $a'=a-\frac{a \trsp b}{b \trsp b}b$. Then 
\[
||a'||_2^2=||a||_2^2-\frac{(a\trsp b)^2}{||b||_2^2},
\]
which implies $||a'||_2\leq\sqrt{15/4}$ as $||a||_2^2=4=||b||_2^2$.

Iterating the replacement of rows from $R$ by shorter rows in $S'$ 
as long as $4k>3n$, 
we may thus assume that 
$R$ contains $k\leq 3n/4$ rows $r$ with $||r||_2=2$,
and that 
$S\cup S'$ contains $n-k\geq n/4$ rows $r$ with $||r||_2\leq \sqrt{15/4}$.
Since 
$$\left(\sqrt{15/4}\right)^{n-k}2^k\leq 
\left(\sqrt{15/4}\right)^{n/4}2^{3n/4}
=\left(\left(\sqrt{15/4}\right)^{1/2}2^{3/2}\right)^{n/2}
\leq 3.936^{n/2},$$ 
Hadamard's inequality (\ref{eh}) implies the desired bound.
\end{proof}

For the $2$-consecutive ones property, 
we can also improve the lower bound of $2^{\frac{n-2}{2}}\approx 1.414^{n-2}$
from Theorem \ref{theoremkcop}.
More precisely, for every positive integer $n$ divisible by $3$,
there is an $n\times n$ $0/1$-matrix $A$ that has the $2$-consecutive one property,
and satisfies 
$\det (A)=4^{\frac{n-3}{3}}\approx 1.587^{n-3}.$
The matrix $A$ is constructed 
 as in the proof of Theorem~\ref{theoremkcop} but with slightly different matrices $H$ and $R$:
\[
H=
\begin{pmatrix}
1 & 1 & -1\\
1 & -1 & 1\\
1 & -1 & -1
\end{pmatrix}\text{ and }
R=\begin{pmatrix}
-1 & 0 & 0\\
0 & 0 & 0\\
0 & 1 & 0
\end{pmatrix}
\]
We believe that this construction is essentially optimal:

\begin{conjecture}\label{conj2cop}
If $A\in \{ 0,1\}^{n\times n}$ has the $2$-consecutive ones property, then 
$|\det(A)|\leq 4^{n/3}.$
\end{conjecture}

\section{Path-edge incidence matrices}\label{pathsec}

In the previous section we have generalised the consecutive ones property by allowing more blocks
of ones. In this section we consider a different generalisation. 

Let $P$ be a path with $m$ edges, and let $\mathcal P$ be a set of some $m$ subpaths of $P$. Then 
the \emph{path-edge incidence matrix} $A(\mathcal P,P)\in\{0,1\}^{\mathcal P\times E(P)}$ defined as 
$$A({\cal P},P)_{Q,e}=
\begin{cases}
1 & \mbox{if $e\in E(Q)$, and}\\
0 & \mbox{otherwise}
\end{cases}$$
has the consecutive ones property (provided the edges are ordered as they appear in $P$).
Conversely, any matrix with the consecutive ones property
is the path-edge incidence matrix for suitable $\mathcal P$ and $P$. In this section we will consider
path-edge incidence matrices $A(\mathcal P,T)$, 
where $\mathcal P$ is a set of subpaths in a tree~$T$ rather than in a path.

We prove:

\begin{theorem}\label{twothm}
If $T$ is a tree on $n$ vertices, and $\mathcal P$ is a set of $n-1$ paths in $T$, then
$$|\det (A(\mathcal P,T))|\leq 2^{n-1}.$$
\end{theorem}

There is some relation between the two generalisations of the consecutive ones property. 
Trivially, if $T$ is the subdivision of a tree with $m$ edges, then $A(\mathcal P,T)$ has the $m$-consecutive ones property. 
Yet, not even all matrices with the $2$-consecutive ones property can be represented 
in the  form $A(\mathcal P,T)$.
As an example, one may easily check that the matrix
\[
\begin{pmatrix}
1 & 1 & 1 & 1\\
1 & 1 & 1 & 0\\
0 & 1 & 1 & 1\\
1 & 0 & 0 & 1
\end{pmatrix}
\]
has the $2$-consecutive ones property but is not of the type $A(\mathcal P,T)$ for any tree $T$ and 
set of subpaths $\mathcal P$.

\begin{proof}[Proof of Theorem~\ref{twothm}]
Select a vertex $r$ in $T$ and consider $T$ as a rooted tree with root $r$. 
We denote by $\leq $ the tree order with minimal element $r$.

We turn the incidence matrix $A(\mathcal P,T)$ into a matrix $A$ by subtracting from 
every column with index $e$ all columns with index $e'$ such that $e$ is the unique ancestor edge
of $e'$ in the tree order; this is done for all $e$ in breadth-first search order,
that is, in non-decreasing tree order.
Obviously, $A(\mathcal P,T)$ and $A$ have the same determinant.

For every $P\in\mathcal P$, the non-zero entries of row $A_{P,\LargerCdot}$ 
fall into one of four cases:
If $e,e'$ are the two leaf-edges of $P$, then
\begin{itemize}
\item if $e$ and $e'$ are incomparable in the tree order and $P$ avoids $r$, then 
$A_{P,\LargerCdot}$ has entry $1$ at $e$ and at $e'$, and entry $-2$ at 
the ancestor edge of the vertex of $P$ that is smallest in the tree order;
\item if $e$ and $e'$ are incomparable but $P$ passes through $r$, then $A_{P,\LargerCdot}$
has entry $1$ at $e$ and at $e'$ (and no entry $-2$);
\item if $e\leq e'$ and $P$ avoids $r$, 
then the row has entry $-1$ at the ancestor edge of $e$ and entry $1$ at $e'$; and
\item if $e\leq e'$ and $P$ ends in $r$, 
then the row has a unique non-zero entry of $1$ at $e'$.
\end{itemize}
Note that at this point, 
Hadamard's inequality (\ref{eh}) already implies 
$$|\det (A(\mathcal P,T))|\leq {6}^\frac{n-1}{2}.$$
In order to obtain our better bound,
we still need to work a bit.

\medskip

\noindent Denote by $F$ the set of edges $e$ such that the column $A_{\LargerCdot,e}$
has an entry $-2$.
For every $P\in\mathcal P$, the row $c=A_{P,\LargerCdot}$
has the following properties:
\begin{enumerate}[(i)]
\item\label{posincomp} All edges $e$ with positive entries $c_e>0$ are incomparable in the tree order.
\item\label{singleneg} $c$ has at most one negative entry, and if $c_e=-2$, then $e\in F$.
\item\label{posneg} If $c$ is negative at edge~$e$ and positive at edge~$f$, then
 $e\leq f$.
\item\label{absval} If $c_{e}\neq -2$, then $|c_{e}|\leq 1$ for every edge~$e$.
\item\label{possum} The sum of all positive entries is at most~$2$, that is, 
$\sum\limits_{e\in E(T):c_e>0}c_e\leq 2$.
\end{enumerate}
We now manipulate the rows of $A$ without changing the determinant and
without losing properties (i)--(v) for any row.
Let $F=\{e_1,\ldots,e_k\}$ and let $i$ be smallest such that the column $A_{\LargerCdot,e_i}$
also contains a positive entry.
 
Suppose that $P$ is a row index with $A_{P,e_i}=-2$ 
and that $Q$ is a row index with $A_{Q,e_i}>0$. 
Set $\alpha=A_{Q,e_i}$ and observe that $\alpha\leq 1$ by~\eqref{absval}.
We replace the row $A_{Q,\LargerCdot}$ 
by a new row $A_{Q,\LargerCdot}+\tfrac{\alpha}{2}A_{P,\LargerCdot}$ 
and denote the new matrix by $A'$. Clearly, $\det (A')=\det (A)$.

We claim that if $A$ had properties (i)--(v), then also $A'$ has these properties. 

Indeed,~\eqref{possum} for $A'$ follows 
from~\eqref{possum} for $A$ as 
the entry $A_{Q,e_i}=\alpha$ decreases to $A'_{Q,e_i}=0$, and
$$
\frac{\alpha}{2}\sum_{e\in E(T):A_{P,e}>0}A_{P,e}\leq \frac{\alpha}{2}\cdot 2= \alpha.
$$
In a similar way,~\eqref{singleneg} for $A'$ follows from~\eqref{singleneg} for~$A$
and $\alpha>0$. 

If $A_{Q,\LargerCdot}$ has no negative entry, then~\eqref{posneg} holds trivially
for $A'_{Q,\LargerCdot}$ as $\alpha>0$.
So, assume there is an edge~$e$ with  $A_{Q,e}<0$. 
It follows that $e$ is strictly smaller in the tree order than any $f$ with $A_{P,f}>0$:
By~\eqref{posneg} for $A_{Q,\LargerCdot}$, the edge $e$ is strictly smaller than $e_i$ as $A_{Q,e_i}>0$,
and, by~\eqref{posneg} for $A_{P,\LargerCdot}$, the edge $e_i$ is smaller than any edge~$f$ with $A_{P,f}>0$
as $A_{P,e_i}<0$. 
This ensures~\eqref{posneg} for $A'$ as $A'_{Q,\LargerCdot}$ can only have a negative entry at~$e$.

Now, let $g$ be an edge of $P$ 
with $A_{P,g}>0$ and let $h\not\in \{ e_i,g\}$ be an edge of $Q$ with $ A_{Q,h}>0$. 
Suppose that $g$ and $h$ are comparable. If $h\leq g$, then as also $e_i\leq g$, 
it follows that $h$ and $e_i$ are comparable, which is impossible by~\eqref{posincomp} for $A_{Q,\LargerCdot}$. 
If, on the other hand, $g\leq h$, then $e_i\leq g$ implies $e_i\leq h$,
which is again impossible. 
Thus, $g$ and $h$ are incomparable. 
This guarantees~\eqref{posincomp} for $A'$.
Since $\tfrac{\alpha}{2}\leq 1$,
it also guarantees~\eqref{absval} 
as it implies that no two positive entries are added when forming the new row~$A'_{Q,\LargerCdot}$.

Note that $A$ and $A'$ have the same columns with $-2$~entries,
namely those with index in $F$, and 
that no column of $A'$ with index in $\{e_1,\ldots, e_{i-1}\}$ contains
a positive entry since this is also the case for $A$. 
Note, moreover, that $A'$ has one fewer positive entry in column $A'_{\LargerCdot, e_i}$,
as $A'_{Q,e_i}=0$.
Thus, by replacing $A$ by $A'$ and iteratively repeating this procedure,
we can eliminate all positive entries in $A_{\LargerCdot,e_i}$. 
Doing this for all $i$ results in a matrix $B$ 
that satisfies~(i)--(v), for which $\det (B)=\det (A)=\det (A(\mathcal P,T))$
and for which
\begin{equation}\label{nonpos}
\text{\em 
no column with index in $F$ contains a positive entry.
}
\end{equation}
Define for such matrices $ B$ two sets of row indices: 
\begin{itemize}
\item $\mathcal R(B)$,  the set of those $P\in\mathcal P$ for which 
an edge $e$ exists such that the column $B_{\LargerCdot,e}$
has a $-2$ at $P$ and zeros everywhere else; and 
\item $\mathcal S( B)$,  the set of those $P\in\mathcal P$ that are~$0$
restricted to $F$ and 
for which $||B_{P,\LargerCdot}||_2\leq 2$.
\end{itemize}
As long as there is some $P\in\mathcal P\setminus (\mathcal R(B)\cup\mathcal S(B))$
such that the row $B_{P,\LargerCdot}$ has an entry of~$-2$, 
we will iteratively modify $B$ so that the determinant does not change, 
properties (i)--(v) are maintained for every row with index not in $\mathcal R(B)\cup\mathcal S(B)$,
and \eqref{nonpos} is maintained
in such a way that $\mathcal R(B)\cup \mathcal S(B)$ grows in each step.

Assume that there is 
$P\in\mathcal P\setminus (\mathcal R(B)\cup\mathcal S(B))$
such that the row $ B_{P,\LargerCdot}$ has an entry of~$-2$ at the edge~$e$, say.
By~\eqref{singleneg}, $e\in F$.
Since $P\notin\mathcal R(B)$, 
there is a second non-zero entry in the column
$B_{\LargerCdot,e}$; let this be in row~$Q\in\mathcal P$. 
By~\eqref{nonpos}, we have $ B_{Q,e}<0$. 
By the definition of $\mathcal R(B)$ and $\mathcal S(B)$, 
it follows that $Q\notin\mathcal R(B)\cup\mathcal S(B)$. 
In particular, (i)--(v) hold for the row $B_{Q,\LargerCdot}$.
Set $\beta=B_{Q,e}$ and replace the row $B_{Q,\LargerCdot}$  
by $B_{Q,\LargerCdot}+\tfrac{\beta}{2}B_{P,\LargerCdot}$.
This results  in a matrix $B'$ of the same determinant as~$B$. 
By~\eqref{singleneg} for the rows $B_{P,\LargerCdot}$ and $B_{Q,\LargerCdot}$, and by \eqref{nonpos},
the new row $B'_{Q,\LargerCdot}$ is~$0$ in $F$, which means that $B'$
satisfies~\eqref{nonpos}.

By~\eqref{absval} for $B_{Q,\LargerCdot}$, 
we have $\beta\in [-2,0)$, which implies, 
together with~\eqref{absval} and~\eqref{singleneg} applied to $P$ and $Q$, 
that every entry of $B'_{Q,\LargerCdot}$ 
has absolute value at most~$1$, that is, $||B'_{Q,\LargerCdot}||_\infty\leq 1$. 
Moreover, from~\eqref{possum} we get
$$
||B'_{Q,\LargerCdot}||_1
\leq\sum_{f:B_{Q,f}>0}B_{Q,f}+\frac{|\beta|}{2}
\sum_{f:B_{P,f}>0}B_{P,f}\leq 2+1\cdot 2=4.
$$
This implies
$$
||B'_{Q,\LargerCdot}||_2\leq\sqrt{||B'_{Q,\LargerCdot}||_1\cdot ||B'_{Q,\LargerCdot}||_\infty}
\leq \sqrt{4\cdot 1}=2,
$$
which means that $Q\in\mathcal S(B')$.
Since the new row $B'_{Q,\LargerCdot}$ is~$0$ in $F$, we have
$\mathcal R(B)\subseteq\mathcal R(B')$. As also 
$\mathcal S(B)\subseteq\mathcal S(B')$,
it follows that 
$\mathcal R(B')\cup\mathcal S(B')$ strictly includes $\mathcal R(B)\cup\mathcal S(B)$.

As the rows of $B'$ with index outside of $\mathcal R(B')\cup\mathcal S(B')$ 
are the same as in $B$, it is clear that such rows still satisfy~(i)--(v). 

\medskip

 Let $C$ be the matrix obtained by a maximal sequence of such modifications,
that is, for every $P\in {\cal P}\setminus (\mathcal R(C)\cup\mathcal S(C))$, 
the row $ C_{P,\LargerCdot}$ has no entry $-2$.
If $P\in {\cal P}\setminus (\mathcal R(C)\cup\mathcal S(C))$, 
then \eqref{absval} holds for $P$, and, hence, $||C_{P,\LargerCdot}||_\infty\leq 1$.
Now,
$$
||C_{P,\LargerCdot}||_2\leq\sqrt{1\cdot ||C_{P,\LargerCdot}||_1}\,
\overset{\eqref{singleneg},\eqref{absval}}{\leq} \sqrt{1+\sum_{f:C_{P,f}>0}C_{P,f}}\leq\sqrt 3,
$$
where the last inequality follows from~\eqref{possum} for $P$. 
As a consequence, we always get  $|| C_{P,\LargerCdot}||_2\leq 2$
for every $P\in\mathcal P\setminus \mathcal R(C)$.
Note that the restriction of $C$ to rows in $\mathcal R(C)$ and columns in $F$
is, after permuting rows and columns, a diagonal matrix with $-2$-entries in the diagonal. 
Thus, expanding the determinant of $C$ along the columns with index in $F$, followed by an application 
of Hadamard's inequality (\ref{eh}), results in $|\det (A(\mathcal P,T))|=|\det (C)|\leq 2^{n-1}$.
\end{proof}

 The bound can be improved by a factor of $4$ 
by choosing the root as a leaf in which at least two paths from $\mathcal P$ end,
which results in two unit vectors in $A(\mathcal P,T)$ along which one can expand the determinant.

We close this section with the construction of a tree $T$ of arbitrarily large order $n$, 
and a set ${\cal P}$ of $n-1$ paths in $T$ 
such that $|\det (A(\mathcal P,T))|=2^{\frac{2}{3}n-\frac{5}{3}}$,
which limits the possible improvements of Theorem~\ref{twothm}.

For any integer $n'$, let $n\geq n'$ such that there is a tree $T$ on $n$ vertices
in which all internal vertices have degree~$3$, and 
in which there is a vertex $r$ that has the same distance, $d$ say, to every leaf. 
Note that $T$ has $n-1=3(2^d-1)$ edges, and consider $T$ as rooted in $r$.

The set $\mathcal P$ of $m$ paths consists of two subsets, $\mathcal P_I$ and $\mathcal P_L$. 
For every edge $uv$ of $T$ between two non-leaves $u$ and $v$ of $T$ such that $u$ is the ancestor of $v$, 
select a leaf-leaf path $P_{uv}$ containing $v$ but not $u$; 
we denote the set of these paths by $\mathcal P_I$. 
Note that $|\mathcal P_I|=3(2^{d-1}-1)=\tfrac{n}{2}-2$.
For the set $\mathcal P_L$ partition the leaves of $T$ into triples $\{ a,b,c\}$
such that $a$, $b$, and $c$ are pairwise separated by the root $r$. 
For each such triple add 
the $a$-$b$~path $Q_{a,b}$, 
the $b$-$c$~path $Q_{b,c}$, and 
the $c$-$a$~path $Q_{c,a}$ to $\mathcal P_L$. 
Note that $|\mathcal P_L|=3\cdot 2^{d-1}=\tfrac{n}{2}+1$.

As in the proof of the theorem, we turn the incidence matrix $A(\mathcal P,T)$ 
into a matrix $A$ by subtracting from each column the columns of direct descendant edges,
where we go through the columns/edges in breadth-first search order. 
Any row of $A$ 
belonging to a path $P_{uv}$ in $\mathcal P_I$ has 
two $+1$-entries at leaf-edges, 
a $-2$-entry at $uv$, 
and~$0$-entries everywhere else. 
Any row belonging to a path $Q_{a,b}$ in $\mathcal P_L$ has 
a $+1$-entry at the leaf-edge incident with $a$,
a $+1$-entry at the leaf-edge incident with $b$, 
and $0$-entries everywhere else. 
In particular, up to a permutation of its rows and columns, 
the matrix $A$ has the following form:

$$
\begin{pmatrix}
\begin{matrix}
-2 & \ldots  & 0\\
\vdots  & \ddots & \vdots \\
0  &    \ldots    & -2 \\
\end{matrix} 
& * \\
0 &
\begin{matrix}
C & \ldots  & 0\\
\vdots  & \ddots & \vdots \\
0  &    \ldots    & C \\
\end{matrix}
\end{pmatrix}, \text{ where }
C=
\begin{pmatrix}
1 & 1 & 0 \\
0 & 1 & 1 \\
1 & 0 & 1 
\end{pmatrix}.
$$
Since $\det (C)=2$, we obtain
$
|\det (A(\mathcal P,T))|=2^{|\mathcal P_I|}\cdot 2^{\frac{1}{3}|\mathcal P_L|}
=2^{\frac{2}{3}n-\frac{5}{3}}.
$

\section{An upper bound on the leaf rank}\label{sec_lr}

Let us give an application of 
Theorem~\ref{twothm}.

There are various ways to capture a graph, or at least its essential properties, 
in terms of a tree. This often helps to understand the graph, since 
trees have  a particularly simple structure. 
While tree decompositions are certainly the most prominent example,  other approaches
have been pursued as well. In the context of  phylogenetic trees, for instance,  
Nishimura et al.~\cite{nirath} proposed the concept of \emph{leaf roots}: 
Given a graph $G$ with vertex set $V(G)$
a tree $T$ is a \emph{$k$-leaf root} of~$G$, for some integer $k$,
if 
\begin{itemize}
\item the set of leaves of $T$ equals $V(G)$, and,
\item any two distinct vertices $u$ and $v$ of $G$ are adjacent in $G$ 
if and only if ${\rm dist}_T(u,v)\leq k$, that is, if their distance in $T$ is at most~$k$.
\end{itemize}

Not all graphs have $k$-leaf roots for any given $k$, or indeed any $k$ at all. 
Brandst\"{a}dt et al.~\cite{bhmw} define the \emph{leaf rank} of a graph $G$:
if there is an integer $k$ such that $G$ has a $k$-leaf root then let the leaf rank of $G$
be the smallest such $k$. If for no $k$ the graph $G$ has a $k$-leaf root then we say that 
$G$ has infinite leaf rank. The $4$-cycle, for instance, has infinite leaf rank.

The structure of graphs with leaf rank at most~$4$ is well understood~\cite{brle,brlera,doguhuni,ra}. 
For leaf rank~$5$ at least an efficient recognition algorithm has been described~\cite{chko},
but for larger leaf rank both a  structural characterisation  and efficient recognition algorithms are  
unknown. 

Even simpler properties are open, such as: How large, in terms of the order, can the leaf rank be (provided
it is finite)? 
While 
Brandst\"{a}dt et al.~\cite{bhmw}
prove that ptolemaic and interval graphs of order $n$ have leaf rank at most~$2n$, 
for general graphs
an upper bound in terms of $n$ on the leaf rank  
was previously known. We prove such a bound:

\begin{theorem}\label{theoremlrleq}
Every graph of order~$n$ has either infinite leaf rank or leaf rank 
 at most $2n2^{2n}$.
\end{theorem}
Is the bound tight? On the contrary, it appears incredibly loose:
We do not know 
of any graph on $n$ vertices that has  finite leaf rank  larger than $n$.
Any polynomial bound in terms of $n$ would be very much welcome.  
\begin{proof}
Let $G$ be a graph  of order $n$, and assume $G$ to have finite leaf rank.
For a tree $T$ whose set of leaves equals the vertex set $V(G)$ of $G$, 
we consider the following system ${\cal I}(T)$ of linear inequalities 
using the $n(T)$ real variables $(\ell_e)_{e\in E(T)}$ and $k$
$$
{\cal I}(T):\,
\begin{cases}
\ell(uTv) -k \leq 0 &\mbox{for every pair $u,v$ of adjacent vertices of $G$, and}\\
\ell(uTv) -k \geq 1 &\mbox{for every pair $u,v$ of non-adjacent vertices of $G$,}
\end{cases}
$$
where $uTv$ is the path in $T$ between the leaves $u$ and $v$ of $T$, and 
$\ell(P)=\sum_{e\in E(P)}\ell_e$ for any path $P$.

If $T$ is a $k$-leaf root of $G$,
then $((\ell_e)_{e\in E(T)},k)$ with $\ell_e=1$ for $e\in E(T)$ is a non-negative integral solution of ${\cal I}(T)$.
That is, the polyhedron 
$$\ph P(T)=\Big\{ ((\ell_e)_{e\in E(T)},k)\in \mathbb{R}^{n(T)}_{\geq 0}:((\ell_e)_{e\in E(T)},k)\mbox{ satisfies }{\cal I}(T)\Big\}$$
is non-empty. Conversely, for any tree $T$ whose set of leaves equals $V(G)$ 
\begin{equation}\label{intpoint}
\text{if $(\ell,k)$ is an integral point of $\ph P(T)$ then the leaf rank of $G$ is at most~$k$.}
\end{equation}

Now, let $T$ be a tree of minimum order such that $\ph P(T)$ is non-empty.
Clearly, the tree $T$ has no vertex of degree $2$,
which implies that the number of its internal vertices is not larger than the number of its leaves.
Consequently
\begin{equation}\label{nbound}
n(T)\leq 2n.
\end{equation}

Suppose that $\ph P(T)$ contains a point $((\ell_e)_{e\in E(T)},k)$ that has a zero entry. 
As $k=0$ is impossible there is then some edge $e$ of $T$ with  $\ell_e=0$.
Contracting the edge $e$ within $T$ yields a smaller tree $T/e$ 
such that $\ph P(T/e)$ is non-empty,
which contradicts the choice of $T$.
Therefore, every point of $\ph P(T)$ has only strictly positive entries. 
By the convexity of polyhedra, this implies that the non-negativity condition is redundant,
that is,
$$\ph P(T)=\Big\{ ((\ell_e)_{e\in E(T)},k)\in \mathbb{R}^{n(T)}:((\ell_e)_{e\in E(T)},k)\mbox{ satisfies }{\cal I}(T)\Big\}.$$
Since $\ph P(T)$ is contained in the positive orthant $\mathbb R_{\geq 0}^{n(T)}$
and as it is non-empty, the polyhedron $\ph P(T)$ must have a vertex $x$.
By a theorem of Hoffman and Kruskal \cite{hk}, 
there is a subsystem ${\cal I}'(T)$ of ${\cal I}(T)$ 
such that  $x$
 is the unique solution of the system of linear equalities obtained by turning the inequalities 
in ${\cal I}'(T)$ into equalities. Since the ambient space is $n(T)$-dimensional, we 
 may assume $|{\cal I}'(T)|=n(T)$.

In view of the structure of ${\cal I}(T)$,
there are two sets ${\cal P}_0$ and ${\cal P}_1$ of altogether $n(T)$ paths in $T$ such that 
$$\{ x\}=\Big\{ ((\ell_e)_{e\in E(T)},k)\in \mathbb{R}^{n(T)}:
\ell(P)-k=i\mbox{ for every $i\in \{ 0,1\}$ and every $P\in {\cal P}_i$}\Big\}.$$
Let us write that in a slightly more concise way. Set
${\cal P}={\cal P}_0\cup {\cal P}_1$ and 
${\bf -1}=\trsp{(-1,\ldots,-1)}\in \mathbb{R}^{n(T)}$. Then there is a
  $0/1$-vector ${\bf b}$ of length $n(T)$ such that  
 $x$ is the unique solution of the system $(A({\cal P},T),{\bf -1})x={\bf b}$ of linear equalities.

Denote by $A_i$ the matrix obtained from $(A({\cal P},T),{\bf -1})$ be replacing the $i$th
column by $\mathbf b$. Then, by Cramer's rule,
\[
x_i=\frac{\det A_i}{\det((A({\cal P},T),{\bf -1}))}.
\]
In particular,  the last entry $k$ of $x$ equals 
\[
k=\frac{\det((A({\cal P},T),{\bf b}))}{\det((A({\cal P},T),{\bf -1}))}.
\]
Thus, $|\det((A({\cal P},T),{\bf -1}))|x$ is an integral point of $\ph P(T)$. Recalling~\eqref{intpoint},
we therefore see that 
\[
k':=|\det((A({\cal P},T),{\bf -1}))| k = |\det((A({\cal P},T),{\bf b}))|
\]
is an upper bound on the leaf rank of $G$.
Expanding the determinant of the matrix $(A({\cal P},T),{\bf b})$ along the last column ${\bf b}$,
we obtain, by Theorem \ref{twothm} and~\eqref{nbound}, that 
$$k'=|\det((A({\cal P},T),{\bf b}))|\leq n(T)2^{n(T)-1}\leq (2n)2^{2n}.$$
\end{proof}

\vfill

\small
\vskip2mm plus 1fill
\noindent
Version \today{}
\bigbreak

\noindent
Henning Bruhn
{\tt <henning.bruhn@uni-ulm.de>}\\
Dieter Rautenbach
{\tt <dieter.rautenbach@uni-ulm.de>}\\
Institut f\"ur Optimierung und Operations Research\\
Universit\"at Ulm\\
Germany\\

\end{document}